\documentclass{article}
\usepackage[T2A]{fontenc}
\usepackage[cp1251]{inputenc}
\usepackage[english]{babel}
\usepackage[tbtags]{amsmath}
\usepackage{amsfonts,amssymb,amscd}
\usepackage{amsthm}
\usepackage[all]{xy}
\makeatletter
\@addtoreset{equation}{subsection}
\makeatother

\newcommand{\KK}{\mathbb{K}}
\newcommand{\LL}{\mathbb{L}}

\newcommand{\CC}{\mathbb{C}}

\newcommand{\PP}{\mathbb{P}}
\newcommand{\QQ}{\mathbb{Q}}

\newcommand{\RR}{\mathbb{R}}

\newcommand{\rk}{\operatorname{rk}}

\newcommand{\Cl}{\operatorname{Cl}}

\newcommand{\Pic }{\operatorname{Pic}}

\newcommand{\Aut}{\operatorname{Aut}}
\newcommand{\Bir}{\operatorname{Bir}}
\newcommand{\Gal}{\operatorname{Gal}}

\newcommand{\Spec}{\operatorname{Spec}}

\usepackage[OT2,T1]{fontenc}
\DeclareSymbolFont{cyrletters}{OT2}{wncyr}{m}{n}
\DeclareMathSymbol{\Sha}{\mathalpha}{cyrletters}{"58}

\newcommand{\comp}\circ

\theoremstyle{plain}
\newtheorem{theorem}[subsection]{Theorem}
\newtheorem{lemma}[subsection]{Lemma}

\newtheorem{proposition}[subsection]{Proposition}

\newtheorem{corollary}[subsection]{Corollary}

\newtheorem*{claim*}{Claim}

\theoremstyle{definition}
\newtheorem{definition}[subsection]{Definition}

\newtheorem*{definition*}{Definition}

\newtheorem{example-remark}[subsection]{Remark-Example}
\newtheorem{subexample-remark}[equation]{Remark-Example}

\newtheorem*{notation*}{Notation}

\newtheorem{remark}[subsection]{Remark}

\newcounter{NN}

\newcounter{NO}

\newcommand{\Address}{{
  \bigskip
  \footnotesize

  A. Avilov, \textsc{National Research University Higher School of Economics, AG Laboratory, HSE, 6 Usacheva str., Moscow, Russia, 119048.}\par\nopagebreak
  \textit{E-mail}: \texttt{v07ulias@gmail.com}

}}

\begin{document}
\title{On forms of the Segre cubic}
\author{Artem Avilov}
\maketitle
\begin{abstract}
In this article we study forms of the Segre cubic over non-algebraically closed fields, their automorphism groups and equivariant birational rigidity. In particular, we show that all forms of the Segre cubic are cubic hypersurfaces and all of them have a point.
\end{abstract}
\section{Introduction}

The Segre cubic is a classical three-dimensional variety with many interesting properties. For example, it is a compactification of the moduli space of configurations of six points on the projective line (see~\cite[\S 2]{Dol}) and its dual variety, called the Igusa quartic, is a compactification of the moduli space of certain abelian surfaces (see~\cite[Theorem 2]{Muk}). Birational geometry of the Segre cubic was extensively studied, for example, its small resolutions were described (see~\cite{Fin}).

The aim of this article to study equivariant birational rigidity of forms of the Segre cubic. Equivariant birational rigidity of the Segre cubic itself over algebraically closed field of characteristic zero was studied by the author in the paper~\cite{Avi}. We show that for every field of characteristic zero there is only one form of the Segre cubic over this field which is $G$-birationally rigid (see Definition~\ref{def3}), and only for the following groups: $S_{6}$, $A_{6}$, $S_{5}$ and $A_{5}$, where $S_{6}$ is the full automorphism group and groups $S_{5}$ and $A_{5}$ are embedded into $S_{6}$ in the standard way (see Definition~\ref{def57}). Moreover, in these cases the form of the Segre cubic is $G$-birationally superrigid. These results can be useful, for example, for classification of finite subgroups of three-dimensional Cremona groups over fields of characteristic zero. We wxpect that these results are valid also over fields of characteristic~\mbox{$p>5$}. Also we prove that all forms of the Segre cubic have a point defined over the basic field, and all of them are cubic hypersurfaces. Special attention is given to the case of the field of real numbers.

In this article we use the following notation for groups: by $C_{n}$ we denote the cyclic group of order $n$; by $D_{2n}$ we denote the dihedral group of order $2n$; by $S_{n}$ we denote the symmetric group of rank~$n$; by $A_{n}$ we denote the alternating group of rank $n$. For an arbitrary field $\KK$ by $\KK^{\operatorname{sep}}$ we denote its separable closure.

This work is supported by the Russian Science Foundation under grant No 18-11-00121. The author is a Young Russian Mathematics award winner and would like to thank its sponsors and jury. The author would like to thank S. Gorchinskiy, C. Shramov and A. Trepalin for useful discussions and comments.
\section{Biregular geometry of the forms of the Segre cubic}
\begin{definition} \emph{The Segre cubic} over the field $\KK$ of characteristic zero or~\mbox{$p\geqslant 5$} is a three-dimensional variety $\mathcal{S}_{\KK}$, which can be explicitely given by the following system of equations in $\PP^{5}_{\KK}$:
\begin{equation}\label{eq}\sum\limits_{i=1}^{6}x_{i}=\sum\limits_{i=1}^{6}x_{i}^{3}=0.
\end{equation}
Usually, if it doesn't lead to misunderstanding, we will omit the index and denote the Segre cubic by $\mathcal{S}$. We will call a variety $X$ defined over an arbitrary field $\KK$ of characteristic zero or $p\geqslant 5$ a \emph{form of the Segre cubic} if $X_{\KK^{\operatorname{sep}}}=X\otimes_{\KK} \Spec\KK^{\operatorname{sep}}$ is isomorphic to the Segre cubic over the field $\KK^{\operatorname{sep}}$. If the field $\KK$ is the field of real numbers $\RR$, we will call the variety $X$ a \emph{real form of the Segre cubic}.
\end{definition}
\begin{remark} Note that equations~\eqref{eq} make sence over an arbitrary field, so there is at least one form of the Segre cubic over an arbitrary field.
\end{remark}
We will use the following well-known facts about the Segre cubic (see, for example,~\cite[\S 2]{Dol}):
\begin{itemize}
\item the automorphism group $\Aut(\mathcal{S})$ is isomorphic to $S_{6}$ and acts by permutations of standard coordinates;
\item the singular set of the Segre cubic $\mathcal{S}$ consists of 10 ordinary double points, all of them form an $\Aut(\mathcal{S})$-orbit and one of them has coordinates $(1:1:1:-1:-1:-1)$;
\item the variety $\mathcal{S}$ contains exactly 15 planes which form an $\Aut(\mathcal{S})$-orbit, one of them can be given by equations $x_{1}+x_{2}=x_{3}+x_{4}=x_{5}+x_{6}=0$ in standard coordinates;
\item every singular point of $\mathcal{S}$ lies on 6 planes and every plane on $\mathcal{S}$ contains exactly 4 singular points. In other words, singular points and planes form a $(10_{6}, 15_{4})$-configuration in notation of~\cite{Dol2}. The automorphism group of such a configuration is isomorphic to $S_{6}$ and is induced by the automorphism group~$\Aut(\mathcal{S})$ (see, for example,~\cite[Lemma 2.1]{Avi2}).
\end{itemize}
\begin{definition}\label{def57} We will say that a subgroup $A_{5}\subset\Aut(\mathcal{S})$ or $S_{5}\subset\Aut(\mathcal{S})$ is \emph{standard} if it preserves some hyperplane $\{x_{i}=0\}$.
\end{definition}
In the sequel we will need the following easy facts about elements of order 2 and certain subgroups of the group $S_{6}$.
\begin{lemma}\label{le1}The element $(1\ 2)$ acting on the $(10_{6}, 15_{4})$-configuration has exactly $4$ fixed singular points and $3$ invariant planes and its centralizer is isomorphic to~$C_{2}\times S_{4}$. The element $(1\ 2)(3\ 4)$ acting on the $(10_{6}, 15_{4})$-configuration has exactly $2$ fixed singular points and $3$ invariant planes and its centralizer is isomorphic to $C_{2}\times (C_{2}^{2}\rtimes C_{2})\simeq C_{2}\times D_{8}$. The element $(1\ 2)(3\ 4)(5\ 6)$ acting on the $(10_{6}, 15_{4})$-configuration has exactly $4$ fixed singular points and $7$ invariant planes and its centralizer is isomorphic to $C_{2}^{3}\rtimes S_{3}$. The stabilizer of a singular point is isomorphic to $S_{3}^{2}\rtimes C_{2}$, the stabilizer of a plane is isomorphic to $S_{4}\times C_{2}$. A non-standard subgroup $S_{5}$ acts on the set of planes with two orbits of length $5$ and $10$ respectively.
\end{lemma}
\begin{proof} Simple direct computations.
\end{proof}
\begin{corollary}\label{cor1} Every real form of the Segre cubic has a singular point defined over the base field $\RR$.
\end{corollary}
\begin{proof} Let $X$ be a real form of the Segre cubic. Consider an action of the complex conjugation $\sigma$ on the $(10_{6}, 15_{4})$-configuration of singular points and planes on the Segre cubic $X_{\CC}$. Since $\sigma$ acts as an element of order~2 or 1, by Lemma~\ref{le1} we have a $\sigma$-invariant singular point.
\end{proof}
For an arbitrary field $\KK$ we have the following assertion.
\begin{lemma}\label{le6} A form $X$ of the Segre cubic over the field $\KK$ contains a $\KK$-point if and only if the variety $X$ is isomorphic to a cubic hypersurface in $\PP^{4}_{\KK}$.
\end{lemma}
\begin{proof} Assume that the variety $X$ has a point defined over the field $\KK$. Then the group $\Pic(X_{\KK^{\operatorname{sep}}})^{\operatorname{Gal}(\KK^{\operatorname{sep}}/\KK)}$ coincides with its subgroup $\Pic(X)$ (see, for example,~\mbox{\cite[Theorem 9.1]{Pro}}). In particular, the divisor class $-\frac{1}{2}K_{X}$ is defined over $\KK$ and its linear system induces an embedding of $X$ into $\PP^{4}_{\KK}$ as a cubic hypersurface. Due to~\cite[Proposition 3.2]{Cor2} the converse statement is also true.
\end{proof}
As a consequence, every real form of the Segre cubic is a cubic itself. Let us show, that analogous result is valid over an arbitrary field of characteristic zero or $p\geqslant 5$. For this purpose we need the following lemma.
\begin{lemma}\label{le11}\textnormal{(cf.~\cite[Lemma 2.3]{Cor2})} Let $\KK$ be an arbitrary field of characteristic different from $2$, and let $\LL/\KK$ be a composite of quadratic extensions of the field $\KK$. Let $X$ be a variety over the field $\KK$ such that $X_{\LL}$ is isomorphic to a cubic in $\PP^{4}_{\LL}$. Suppose that $X_{\LL}$ has an $\LL$-point. Then $X$ has a $\KK$-point.
\end{lemma}
\begin{proof} There is a sequence of quadratic field extensions $$\KK=\LL_{0}\subset\LL_{1}\subset\LL_{2}\subset ... \subset \LL_{n}=\LL.$$ By the induction it is enough to consider the case when $\LL$ is a quadratic extension of $\KK$. Since the characteristic of $\KK$ differs from $2$, the extension $\KK\subset \LL$ is a Galois extension. We know that the cubic $X_{\LL}$ contains an $\LL$-point. This point is either defined over $\KK$ or its $\Gal(\LL/\KK)$-orbit consists of two points. In the second case the line $l$ passing through these points is defined over $\KK$. If this line lies on $X_{\LL}$ then $X$ contains every $\KK$-point of the line $l$. If the line $l$ doesn't lie on $X_{\LL}$ then the third intersection point of the line $l$ and the cubic $X_{\LL}$ is defined over $\KK$.
\end{proof}
\begin{lemma}\label{le14} Let $\KK$ be a field of characteristic zero or $p\geqslant 5$. Then every form of the Segre cubic over the field $\KK$ has a $\KK$-point and isomorphic to a cubic hypersurface.
\end{lemma}
\begin{proof} Let $X$ be a form of the Segre cubic. There is the exact sequence $$0\to \Pic(X)\to \Pic(X_{\KK^{\operatorname{sep}}})^{\Gal(\KK^{\operatorname{sep}}/\KK)}\to \operatorname{Br}(\KK),$$
see, for example,~\cite[Problem 3.3.5(iii)]{GSh}. If the divisor class $-\frac{1}{2}K_{X}$ is not defined over the field $\KK$, then the group $\Pic(X)$ is embedded into $\Pic(X_{\KK^{\operatorname{sep}}})^{\Gal(\KK^{\operatorname{sep}}/\KK)}$ as a subgroup of index 2, so we define canonically an element of order 2 in the Brauer group $\operatorname{Br}(\KK)$. Since this element has a representation as a tensor product of quaternion algebras over the field $\KK$ (see~\cite[Theorem 11.5]{MeS}), there is a composite $\LL$ of quadratic extensions of the field $\KK$ such that our algebra splits over this extension. Thus the corresponding embedding $$\Pic(X_{\LL})\to \Pic(X_{\KK^{\operatorname{sep}}})^{\Gal(\KK^{\operatorname{sep}}/\LL)}$$ is an isomorphism. The linear system $|-\frac{1}{2}K_{X_{\LL}}|$ induces an embedding of $X_{\LL}$ into~$\PP^{4}_{\LL}$ as a cubic hypersurface. By Lemma~\ref{le6} there is an $\LL$-point on $X_{\LL}$. By Lemma~\ref{le11} there is a $\KK$-point on $X$. In particular, by Lemma~\ref{le6}, the variety $X$ is isomorphic to a cubic hypersurface in $\PP^{4}_{\KK}$.
\end{proof}
Now we return to the field of real numbers.
\begin{proposition}\label{pr3} There are exactly $4$ real forms of the Segre cubic $\mathcal{S}$ up to isomorphism. All of them are three-dimensional cubic hypersurfaces and rational over the field $\RR$.
\end{proposition}
\begin{proof} It is well-known (see, for example,~\cite[Chapter III, \S 1]{Ser}) that there is a one-to-one correspondence between the forms of a real variety~$X$ and the elements of $\operatorname{H}^{1}(\operatorname{Gal}(\CC/\RR), \Aut(X_{\CC}))$. The latter set, in turn, is in one-to-one correspondence with the set of all homomorphisms $$\operatorname{Gal(\CC/\RR)}\simeq C_{2}\to S_{6}\simeq\Aut(X_{\CC})$$ up to conjugation. Such homomorphisms are defined by conjugacy classes of elements of order 2 or 1 in the group $S_{6}$. There are exactly four such classes: the trivial permutation and conjugacy classes of the transposition $(1\ 2)$, the product of two transpositions $(1\ 2)(3\ 4)$ and the product of three transpositions $(1\ 2)(3\ 4)(5\ 6)$.

 Let $X$ be a real form of the Segre cubic $\mathcal{S}$. By Corollary~\ref{cor1} the variety $X$ contains an $\RR$-point. By Lemma~\ref{le6} the variety $X$ is isomorphic to a cubic in $\PP^{4}_{\RR}$. This cubic contains a singular point defined over the field $\RR$ and the projection from such a point gives us a birational map from the variety $X$ to $\PP^{3}_{\RR}$.
\end{proof}
We say that a real form $X$ of the Segre cubic is of type I (resp., II, III or IV) if the image of the complex conjugation in the group $\Aut(\mathcal{S})\simeq S_{6}$ is the trivial permutation (resp., is conjugate to the transposition $(1\ 2)$, is conjugate to the permutation $(1\ 2)(3\ 4)$ or is conjugate to the permutation $(1\ 2)(3\ 4)(5\ 6)$).
\begin{remark}\label{rem} The form of the Segre cubic of type I can be obtained in the following way: blow up five $\RR$-points of $\PP^{3}_{\RR}$ in general position and contract 10 proper transforms of lines passing through pairs of points. The obtained variety is the required form of the Segre cubic. It also can be defined explicitely by the equations~\eqref{eq}. Note that in this case there is an action of the group $S_{5}$ on $\PP^{3}$ with five marked points and the construction is $S_{5}$-equivariant. The forms of type IV and III can be constructed in a similar way, but we need to blow up $\PP^{3}_{\RR}$ in 3 real points and one pair of conjugated points or one real point and two pairs of conjugated points in general position respectively. One can easily see that we obtain exactly forms of type IV and III by calculation of numbers of singular $\RR$-points and planes defined over $\RR$ in both cases. The form of type II has no such a transformation into $\PP^{3}_{\RR}$, but by Proposition~\ref{pr3} a transformation of other type exists.
\end{remark}
In the following proposition we describe automorphism groups of all forms of the Segre cubic.
\begin{proposition}\label{pr4} The automorphism group of a form of the Segre cubic over an arbitrary field $\KK$ of characteristic zero or $p\geqslant 5$ coincides with the centralizer of the image of the Galois group $\Gal(\KK^{\operatorname{sep}}/\KK)$ in the automorphism group of the $(10_{6}, 15_{4})$-configuration of singular points and planes on $\mathcal{S}$. The automorphism group of the form of the Segre cubic of type $\mathrm{I}$ (resp., $\mathrm{II}$, $\mathrm{III}$ or $\mathrm{IV}$) is isomorphic to $S_{6}$ (resp., $C_{2}\times S_{4}$, $C_{2}\times D_{8}$ or $C_{2}^{3}\rtimes S_{3}).$
\end{proposition}
\begin{proof} Let $X$ be a form of the Segre cubic over $\KK$. There is a canonical embedding of the automorphism group of $X$ into the automorphism group of the $(10_{6}, 15_{4})$-configuration of singular points and planes on $X_{\KK^{\operatorname{sep}}}\simeq\mathcal{S}$ and the canonical map from the Galois group $\Gal(\KK^{\operatorname{sep}}/\KK)$ into the same group $S_{6}$. Since all automorphisms defined over the base field commute with the action of the Galois group, the image of such embedding is contained in the centralizer of the image of the Galois group $\Gal(\KK^{\operatorname{sep}}/\KK)$.

  Let $g\in S_{6}$ be an element of the automorphism group of the $(10_{6}, 15_{4})$-configura\-tion which commutes with the image of some element $\sigma\in\Gal(\KK^{\operatorname{sep}}/\KK)$. Let $\widetilde{g}$ be the corresponding automorphism of the variety $X_{\KK^{\operatorname{sep}}}$. Then $\widetilde{g}^{-1}\circ\sigma^{-1}\circ\widetilde{g}\circ\sigma$ is a linear transformation of $\PP_{\KK^{\operatorname{sep}}}^{4}$ which fixes all singular points of the variety $X_{\KK^{\operatorname{sep}}}$. Since there are 10 singular points and they are in general enough position this linear map is trivial. Thus $\sigma^{-1}\circ\widetilde{g}\circ\sigma=\widetilde{g}$. If $g$ commutes with all elements of the image of $\Gal(\KK^{\operatorname{sep}}/\KK)$ then $\widetilde{g}$ is defined over the base field $\KK$. As a consequence, the embedding of the group~$\Aut(X)$ into the centralizer of the image of the Galois group is an isomorphism.

The second assertion is a consequense of Lemma~\ref{le1}.
\end{proof}
\section{Birational geometry of the forms of the Segre cubic}

For the classification of finite subgroups in Cremona groups it is important to study $G$-birational rigidity of Fano varieties.
\begin{definition} Let $X$ and $Y$ be a varieties with an action of a finite group $G$. We call a rational map $f:X\dasharrow Y$ a \emph{$G$-equivariant map} if there exist an automorphism $\tau$ of the group $G$ such that the following diagram commutes for every $g\in G$:
$$\xymatrix{
X \ar@{-->}[r]^{f}\ar[d]^{g}& Y \ar[d]^{\tau(g)}
\\
X \ar@{-->}[r]^{f}& Y
}$$
We denote the group of $G$-equivariant automorphisms of a $G$-variety $X$ by $\Aut^{G}(X)$ and the group of $G$-equivariant birational selfmaps of a $G$-variety $X$ by $\Bir^{G}(X)$.
\end{definition}
\begin{definition}\label{def3} Let $G$ be a finite subgroup of the automorphism group of a Fano variety $X$ with terminal singularities, and suppose that $X$ is a $G\QQ$-factorial variety and $\rk\Pic(X)^{G}=1$. The variety $X$ is called \emph{$G$-birationally rigid} if there is no birational map from $X$ to another $G$-Mori fibration. If one also has $\Bir^{G}(X)=\Aut^{G}(X)$ then $X$ is called \emph{$G$-birationally superrigid}.
\end{definition}

There is the following theorem:
\begin{theorem}\label{th57}\textnormal{(see~\cite[Proposition 4.1, Theorem 4.8]{Avi})} Let $G\subset\Aut(\mathcal{S})$ be a subgroup of the automorphism group of the Segre cubic over an algebraically closed field of characteristic zero. Then the variety $\mathcal{S}$ is  $G$-birationally rigid if and only if $G$ containts a standard subgroup $A_{5}\subset \Aut(\mathcal{S})$. Moreover, in this case $\mathcal{S}$ is $G$-birationally superrigid.
\end{theorem}
The next theorem is an analog of the previous statement over an arbitrary field of characteristic zero.
\begin{theorem} Let $X$ be a form of the Segre cubic over some field $\KK$ of characteristic zero, and let $G\subset\Aut(X)$ be a subgroup. Assume that \mbox{$\rk\Pic(X)^{G}=1$} and that the variety $X$ is $G\QQ$-factorial and $G$-birationally rigid. Then the variety $X$ can be explicitely given by equations~\eqref{eq} and $G$ contains a standard subgroup~$A_{5}\subset S_{6}$. Conversely, the variety given by the equations~\eqref{eq} is $A_{5}$-birationally superrigid with respect to a standard subgroup $A_{5}\subset\Aut(X)$.
\end{theorem}
\begin{proof} As were noticed earlier, there is a canonical embedding~$G\subset S_{6}$ where $S_{6}$ is the automorphism group of the $(10_{6}, 15_{4})$-configuration. Denote by $H\subset S_{6}$ the image of the Galois group $\Gal(\overline{\KK}/\KK)$ in $S_{6}$. By Proposition~\ref{pr4} the group $G$ is contained in the centralizer of the group $H$. Let $F\subset S_{6}$ be a subgroup generated by $G$ and~$H$.

 If $F$ does not contain a standard subgroup $A_{5}\subset S_{6}$ then it is not hard to check directly with computer that the group $F$ lies in one of the following subgroups: non-standard subgroup $S_{5}\subset S_{6}$, $S_{3}^{2}\rtimes C_{2}$ (which is the stabilizer of a point), $S_{4}\times C_{2}$ (which is the stabilizer of a plane on $\mathcal{S}$) or $S_{4}\times C_{2}$ (which is conjugate to the following group: $S_{4}$ acts by permutations of coordinates $x_{1}, x_{2}, x_{3}, x_{4}$ while~$C_{2}$ permutes $x_{5}$ and $x_{6}$ in standard coordinates). Consider an action of a non-standard subgroup $S_{5}\subset S_{6}$ on the set of planes on $\mathcal{S}$. By Lemma~\ref{le1} it has two orbits of length 5 and 10 respectively. The sum of all planes in the first orbit is not a $\QQ$-Cartier divisor since the group  $\Pic(X)$ is a primitive sublattice in the group $\Cl(X)$ and is generated by the class of hyperplane section while the sum of 5 planes is not a integer multiple of a hyperplane section in $\Cl(X)$. Thus in this case the $G$-variety is not $G\QQ$-factorial. In the second case the projection from a singular point which is invariant with respect to $F$ gives us an equivariant birational map to~$\PP^{3}_{\KK}$. In the third case an $F$-invariant plane is a Weil divisor which is not $\QQ$-Cartier by the same reason as in the first case. But this is impossible since we assume that the variety $X$ is $G\QQ$-factorial.

  Let us consider the forth case. Let the group $S_{4}\times C_{2}$ act on $\PP_{\overline{\KK}}^{4}$ as was described above. Then there is an $S_{4}\times C_{2}$-orbit that consists of the following planes on $\mathcal{S}$:
  $$\begin{aligned}
  & x_{1}+x_{2}=x_{3}+x_{4}=x_{5}+x_{6}=0,\\
  & x_{1}+x_{3}=x_{2}+x_{4}=x_{5}+x_{6}=0,\\
  & x_{1}+x_{4}=x_{2}+x_{3}=x_{5}+x_{6}=0.
  \end{aligned}
  $$
    They form a hyperplane section of $\mathcal{S}$ given by the equation $x_{5}+x_{6}=0$. The only common point of these three planes $p=(0:0:0:0:1:-1)$ is defined over $\KK$ and is $G$-invariant. Since the action of the group $G$ on the hyperplane is a projectivisation of a four-dimensional representation of the group $G$ and the invariant point $p$ correspondes to a one-dimensional subrepresentation, we can find a three-dimensional subrepresentation of $G$ as well. The corresponding $G$-invariant plane is defined over the base field. The projection from such a plane gives us a structure of $G$-equivariant cubic fibration. Now we can apply a $G$-equivariant resolution of singularities and relative $G$-equivariant minimal model program and obtain a $G$-birational transformation into a $G$-Mori fibration with the base of positive dimension.

 So the group $F$ contains a standard subgroup $A_{5}\subset S_{6}$, hence it is isomorphic to one one the following groups: $A_{5}$, $S_{5}$, $A_{6}$ or $S_{6}$. The groups $G$ and $H$ are normal subgroups of $F$, all elements of $G$ commute with all elements of $H$ and they generate the whole group $F$. Obviously, it is possible only in the following case: one of the groups $G$ and $F$ coincides with $F$ while the other one is trivial. If the group $G$ is trivial then the projection from any plane in $\PP^{4}$ which is defined over the base field gives us a structure of fibration by cubic surfaces. If the group~$H$ is trivial then the variety $X$ can be given by the equations~\eqref{eq} and the group~$G$ contains a standard subgroup $A_{5}\subset S_{6}$.

Conversely, assume that the form $X$ of the Segre cubic is given by the equations~\eqref{eq} and the group $G$ contains a standard subgroup $A_{5}\subset S_{6}\simeq \Aut(X)$. Then by~\cite[Lemma 4.6, Lemma 4.7]{Avi} the pair $(X, \frac{1}{\mu}\mathcal{M})$ is canonical for every $\mu$ and every movable $G$-invariant linear subsystem~\mbox{$\mathcal{M}\subset |-\mu K_{X}|$}. Hence by the Noether--Fano inequalities (see, for example,~\cite[Theorem 2.4]{Cor} and~\cite[Theorem 3.2.6]{ChS} in a $G$-equivariant situation) the variety $X$ is $G$-birationally superrigid.
\end{proof}
\begin{corollary} Among all real forms of the Segre cubic only the form of type~$\mathrm{I}$ can be $G$-birationally rigid and only if the group $G$ contains a standard subgroup~$A_{5}$. In this case it is $G$-birationally superrigid.
\end{corollary}
\begin{remark} It looks like the analogous statement (at least in one direction) should be valid also for fields of characteristic $p>5$. We have a minimal model program for threefolds over algebraically closed field of characteristic $p>5$ (see~\cite{HaX}) and it should work also in relative situation with a group action over arbitrary perfect field, but it is not known to the author if it is written down -nywhere. To prove the converse statement we need an analog of the Noether--Fano inequalities in positive characteristic and the existence of such analog is also unknown to the author.
\end{remark}
\begin{remark} It is still an open question: does the birational (super)rigidity of the variety $X_{\overline{k}}=X\otimes_{\Spec k} \Spec \overline{k}$ over the field $\overline{k}$ implies the birational (super)rigidity of the variety $X$ over the field $k$ and is the same assertion true for varieties with the group action (see discussion in~\cite{Kol} and especially~\cite[Question 4]{Kol}). Such a result is valid for algebraically closed field $k$ of characteristic zero and its algebraically closed extension $K$, see~\cite[Theorem 2, Theorem 6]{Kol}.
\end{remark}

\Address
\end{document}